\documentclass[11pt]{amsart}
\usepackage[utf8]{inputenc}
\usepackage{amssymb}
\usepackage{amsmath}
\usepackage{mathtools}  
\usepackage{tabulary}
\usepackage{booktabs}
\usepackage{setspace}
\usepackage{environ}
\usepackage{mathrsfs}
\usepackage{amsfonts}
\usepackage{mathabx}
\usepackage{pict2e}
\usepackage{enumitem}
\usepackage{tikz}
\usepackage[english]{babel}
\usepackage{amsthm}
\usepackage[english]{babel}
\usepackage[numbers]{natbib}
\usepackage{float}
\usepackage{mathpazo}

\usepackage{geometry}
\usepackage{graphicx}

\NewDocumentCommand{\mgn}{}{\mathcal M_{g,n}}
\NewDocumentCommand{\tgn}{}{\mathcal T_{g,n}}
\NewDocumentCommand{\mgnb}{}{\overline{\mathcal M}_{g,n}}
\NewDocumentCommand{\teich}{}{\text{Teichm\"uller}}

\DeclareMathOperator{\cp}{Cap}
\DeclareMathOperator{\kiss}{Kiss}
\DeclareMathOperator{\ssh}{SSH}

\DeclareMathOperator{\ind}{ind}
\DeclareMathOperator{\sys}{sys}

\DeclareMathOperator{\spn}{Span}

\DeclareMathOperator{\rank}{rank}
\DeclareMathOperator{\syst}{sys_T}

\newtheorem{theorem}{Theorem}[section]
\newtheorem*{theorem*}{Theorem}
\newtheorem*{convention*}{Convention}

\newtheorem{lemma}[theorem]{Lemma}

\newtheorem{corollary}[theorem]{Corollary}

\theoremstyle{definition}
\newtheorem{definition}[theorem]{Definition}
\newtheorem*{definition*}{Definition}

\theoremstyle{remark}
\newtheorem{remark}[theorem]{Remark}

\numberwithin{equation}{section}

\title{Index gap of the systole function}
\author{Changjie Chen}

\begin{document}

\begin{abstract}
    It is known that the systole function is topologically Morse on the moduli space $\mgn$ and the $\syst$ functions are $C^2$-Morse on the Deligne-Mumford compactification $\mgnb$. In this paper, We show that these Morse functions admit an index gap on $\mgn$. Specifically, there exists a universal constant $C>0$ such that any critical point in $\mgn$ has Morse index at least $C\log\log(g+n)$. This implies by Morse theory that the low degree homology of the Deligne-Mumford compactification $\mgnb$ comes from the boundary $\partial\mgn$.
\end{abstract}

\maketitle

\section{Introduction}

The moduli space $\mgn$ of Riemann surfaces of genus $g$ with $n$ punctures is a fundamental and popular object in many areas to study. It encodes information in algebraic geometry, differential geometry, and topology, bridging several mathematical branches, and has rich algebraic, topological, and geometric structures.

The \textit{systole} function takes the value on a hyperbolic surface $X$ of the length of a shortest closed geodesic. Akrout showed in \cite{akrout2003singularites} that the systole function is topologically Morse on any moduli space $\mgn$. This gives hope for the use of Morse theory on moduli spaces. However, this idea is not realizable, because the function is not differentiable and the base space is not compact. See \cite{milnor2016morse} for Morse theory.

In \cite{chen2023c}, the author introduces a family of functions: $$\syst(X):=-T\log\sum_{\gamma \text{ s.c.g. on } X} e^{-\frac1Tl_\gamma(X)},$$ where s.c.g. stands for simple closed geodesic. The $\syst$ functions are Morse on $\mgnb$ and behave well with the Weil-Petersson metric, cf. Theorem \ref{syst}, allowing one to apply Morse theory to study $\mgn$ and $\mgnb$.

It has also been shown that in $\mgn$, the set of $\sys$-critical points, the set of $\syst$-critical points, and the set of \textit{eutactic} points, are canonically isomorphic. Therefore, one can unify the term \textit{index} for critical points for either function or for a eutactic point.

In \cite{schaller1999systoles}, Schmutz Scaller constructed a critical point of index $2g-1$ in every $\mathcal M_{g,0}$ and took the guess that the lowest index of the systole function on $\mathcal M_{g,0}$ is $2g-1$. This was disproved by Fortier Bourque in \cite{bourque2020hyperbolic}, by showing the existence of a genus $g>0$, and a critical point in $\mathcal M_{g,0}$, of index at most $\epsilon g$ for arbitrarily small $\epsilon>0$. It remains unknown whether critical points of the Morse functions on $\mgn$ all have Morse indices above some positive constant. The existence of such a positive lower bound would provide insight into the topological relationship between strata of $\mgnb$.

The goal of this paper is to prove that the lowest index can be bounded from below by a positive number that goes to infinity as the dimension of the moduli space increases. This appears to be the first method in the literature to study the lowest Morse index of critical points, providing the first bound on the growth rate.

\begin{theorem}[= \textbf{Theorem \ref{main theorem down}}]
    For any $k\ge0$, with finitely many exceptions on $(g,n)$, all critical points have Morse index greater than $k$ in $\mgn$, for the systole function or the $\syst$ functions.
\end{theorem}

Equivalently, by properties of the $\syst$ functions, all low index critical points of $\syst$ live in the Deligne-Mumford boundary. This implies that the low degree (co)homology of $\mgnb$ comes from the boundary $\partial\mgn$. More precisely, the inclusion $i\colon\partial\mgn\to\mgnb$ induces an isomorphism $i_*\colon H_k(\partial \mgn;\mathbb Q) \xrightarrow{\cong} H_k(\mgnb;\mathbb Q)$ between low degree (co)homology groups, cf. \cite{chen2023c}.

More effectively, we are able to estimate the growth rate of the index gap.

\begin{theorem}[= \textbf{Theorem \ref{kissingnumbersubsurface}}]
    There exists a universal constant $C$ such that the lowest index of the systole function and $\syst$ functions on $\mgn$ is at least $C\log\log(g+n)$.
\end{theorem}

Following these results, in the last section, we will give a full classification of critical points of low Morse indices.

We will show a general result on the rank of a set of (Weil-Petersson) gradient vectors of geodesic-length functions, which eventually implies the first theorem above by eutacticity and is also interesting in its own right.

Let a $j$-\textit{system} on a hyperbolic surface be a set of simple closed geodesics that pairwise intersect each other at most $j$ times.

\begin{theorem}[= \textbf{Theorem \ref{induction}}]
Given $j\ge0$, for any $k\ge0$, there exists $r_k>0$ such that, with finitely many exceptions on $(g,n)$, for any $j$-system $S$ with $\#S=r_k$, for any $\teich$ space $\tgn$ and on any surface $X\in\tgn$, at $T_X\tgn$, there is
$$\rank\{\nabla l_\gamma\}_{\gamma\in S}\ge k.$$
\end{theorem}

We also give an effective upper bound on the cardinality of a $j$-system on a $(g,n)$-surface, although the growth rate is weaker than known results. This was first studied by Juvan, Malni\v{c} and Mohar in \cite{juvan1996systems}, where it was shown that the maximum is finite. Malestein, Rivin and Theran gave explicit bounds for $j=1$ in \cite{malestein2014topological}. Przytycki significantly improved the bound by giving a sharp upper bound for 1-system of arcs in \cite{przytycki2015arcs} and studied the asymptotic growth for general $j$. Aougab, Biringer, Gaster improved the estimate on the asymptotic growth in \cite{aougab2019packing}. Greene improved the bound for 1-system on closed surfaces in \cite{greene2018curves}.

Let $M(g,n)$ be the maximum cardinality of a minimal filling set on a $(g,n)$-surface.

\begin{theorem}[= \textbf{Theorem \ref{maxm} and \ref{maxcap}}]
    We have the following:
    
    (1) $M(0,2)=1.$

    (2) $M(g,n)\le 3g+n.$

    (3) Let $S$ be a $j$-system on a $(g,n)$-surface, then $$\#S\le M(g,n)+(2jM(g,n)(M(g,n)-1))^{jM(g,n)}.$$
\end{theorem}

The organization of this paper is as follows. In Section \ref{Morse section}, we review the systole function and $\syst$ functions, as well as eutacticity. We study topological properties of $j$-systems in Section \ref{basicdefinitions}, and geometric properties in Section \ref{essentialsection}, where we show the rank statement. We apply the rank statement on sets of shortest geodesics in Section \ref{systole section} and obtain the first theorem above. We compute for a lower bound on the lowest index in Section \ref{logloglowerbound}. In Section \ref{classification}, we classify all critical points of index 0, 1 and 2.

\section{Morse functions}
\label{Morse section}
We review theorems on the systole function, and the $\syst$ functions. For background of hyperbolic geometry and $\teich$ theory, one may refer to \cite{buser2010geometry} or \cite{imayoshi2012introduction}.

\begin{definition}
Let $0<T<1$ be a parameter. For a hyperbolic surface $X\in\mgn$,
$$\sys(X):=\min_{\gamma\text{ closed geodesic on }X}l_\gamma(X),$$ and
$$\syst(X):=-T\log\left(\sum_{\gamma \text{ s.c.g. on } X} e^{-\frac1Tl_\gamma(X)}\right),$$
where s.c.g. stands for simple closed geodesic. For a nodal surface $X\in\partial\mgn$ with $s$ nodes, $$\syst(X):=-T\log\left(s+\sum_{\gamma \text{ s.c.g. on } X} e^{-\frac1Tl_\gamma(X)}\right).$$
\end{definition}

We let $S(X)$ denote the set of shortest closed geodesics on a hyperbolic surface $X$.

\begin{definition}[Eutacticity]
    A point $X\in\tgn$ is called \textit{eutactic}, if in the tangent space $T_X\tgn$, the origin is contained in the interior of the convex hull of $\{\nabla l_\gamma\}_{\gamma\in S(X)}$, the set of (Weil-Petersson) gradient vectors of the geodesic length functions associated to the shortest geodesics, with respect to the subspace topology on $\spn\{\nabla l_\gamma\}_{\gamma\in S(X)}$.
\end{definition}

\begin{definition}[Topological Morse function, cf. \cite{morse1959topologically}]
Let $M$ be an $n$-dimensional manifold and $f\colon M\to\mathbb R$ a continuous function.

(1) A point $x\in M$ is called ($C^0$-)\textit{regular} if on a $C^0$-chart around $x$, $f$ is a coordinate function, otherwise it is called ($C^0$-) critical.

(2) A \textit{critical point} $x$ is nondegenerate if on a $C^0$-chart $(x^i)$ around $x$ such that $$f-f(x)=(x^1)^2+\cdots+(x^r)^2-(x^{r+1})^2-\cdots(x^n)^2.$$ In this case the \textit{index} of $f$ at $x$ is defined to be $n-r$.

(3) A continuous function is called \textit{topologically Morse} if all the critical points are nondegenerate.
\end{definition}

\begin{theorem} [\cite{akrout2003singularites}]
\label{Akrout's theorem}
The systole function is topologically Morse on $\mgn$ for any $(g,n)$. $X\in\mgn$ is a critical point if and only if $X$ is eutactic, and in this case the index is equal to $\rank\{\nabla l_\gamma\}_{\gamma\in S(X)}$.
\end{theorem}

We have the following theorem on the $\syst$ functions.

\begin{theorem}[\cite{chen2023c}]
\label{syst}
    (I) For sufficiently small $T>0$, $\syst$ is a $C^2$-Morse function on the Deligne-Mumford compactification $\overline{\mathcal M}_{g,n}$.

    (II) In $\mgn$, $\syst$-critical points and eutactic points are in pairs, in each of which the two points, say $p_T$ and $p$, have Weil-Petersson distance $d_{WP}(p_T,p)<CT$, and the Morse index of $p_T$ is equal to the eutactic rank of $p$. Consequently, the former converges to the latter as $T\to0^+$.

    (III) Let $\mathcal S$ be a boundary stratum of $\mgnb$ that is canonically isomorphic to the product $\prod \mathcal M_i$ of moduli spaces, and $X=\coprod_{\{\text{nodes}\}} X_i\in\mathcal S$ be the decomposition with respect to the stratification such that $X_i\in\mathcal M_i$. The following are equivalent:

    (1) $X$ is a critical point of $\syst$ on $\mgnb$;
    
    (2) $X_i$ is a critical point of $\syst$ on $\mathcal M_i$, for all $i$.
    
    And in this case, $$\ind(X)=\sum\ind(X_i).$$

    (IV) The Weil-Petersson gradient flow of $\syst$ on $\overline{\mathcal M}_{g,n}$ is well defined. Each gradient flow line stays in a single stratum and flows down from a critical point to a critical point in the same stratum or converges to a critical point in its boundary stratum.
\end{theorem}

\begin{remark}
\label{remark rank statement}
    The rank statement in both the theorem on $\sys$ and $\syst$ is what we will use to estimate the index at a critical point.
\end{remark}

In particular, (III) gives a description of the critical points in the Deligne-Mumford boundary $\partial\mathcal M_{g,n}$, according to which we can decompose such a critical point into smaller surfaces that are critical points in their respective moduli spaces. We can also construct a critical point by connecting critical points in any moduli space by nodes. As a result, the study of critical points in the Deligne-Mumford compactification can be reduced to each smaller moduli space.

Since the Morse index is compatible with the stratification, the result on low index critical points in $\mgn$ comes with the result on low index critical points in $\mgnb$, which by Morse theory, gives the consequent (co)homological conclusion on $\mgnb$, see \cite{chen2023c}.

\section{Curves on subsurfaces}
\label{basicdefinitions}

We investigate a set of curves that satisfy certain intersection conditions, on any subsurface of a hyperbolic surface.

\begin{convention*}
    In the following sections, we will utilize specific functions on the space of hyperbolic surfaces that are invariant under diffeomorphisms or hyperbolic isometries. If $P$ is such a function, then we may also use $P(g,n)$ for $P(X)$ for any $[g,n]$-surface $X$, by abuse of notation.
\end{convention*}

\begin{definition}
    A $(g,n)$-surface is a complete hyperbolic surface of genus $g$ with $n$ punctures. A $(g,n,b)$-surface is a hyperbolic surface of genus $g$ with $n$ punctures and $b$ geodesic boundary components. A $[g,n]$-surface is a hyperbolic surface of genus $g$ with a total number $n$ of punctures and geodesic boundary components. For convenience, we use $[0,2]$-surface to refer to a circle or an annulus.
\end{definition}

\begin{definition}
A \textit{subsurface} of a hyperbolic surface $X$ is some $[g,n]$-surface whose interior is isometrically embedded in $X$. The \textit{subsurface hull} $\ssh(S)$ of a set $S$ of simple closed geodesics on $X$ is the minimal subsurface of $X$ that contains $\bigcup S$.
\end{definition}

\begin{remark}
    The definition is valid in terms of uniqueness of such minimal subsurface. If two subsurfaces $X_1$ and $X_2$ intersect, there is the unique subsurface $X_0$ that `supports' $X_1\cap X_2$ by pulling straight the piecewise geodesic boundaries. More precisely, if a simple closed geodesic $\gamma\subset X_1\cap X_2$, then $\gamma\subset X_0$.
\end{remark}

\begin{definition}
    Let $S$ be a set of simple closed geodesics on a hyperbolic surface possibly with geodesic boundary, then $S$ is called
    
    (1) a $j$\textit{-system}, if curves in $S$ intersect pairwise at most $j$ times;

    (2) \textit{filling}, if each complementary region is a hyperbolic polygon, once-punctured polygon, or once-holed polygon, where the hole comes from a boundary component;

    (3) \textit{minimal filling}, if no proper subset is filling.
\end{definition}

\begin{remark}
    (1) $S$ is filling $X$, if and only if every non-peripheral curve on $X$, that is not free homotopic to a boundary component, can be homotoped onto $\bigcup S$.
    
    (2) $S$ is minimal filling if and only if for any $\gamma\in S$, $S\setminus\{\gamma\}$ is not filling.
\end{remark}

\begin{lemma}
\label{fillingconnected}
    Let $S$ be a filling set of simple closed geodesics on a (connected) hyperbolic surface, then $\bigcup S$ is connected.
\end{lemma}

\begin{proof}
    Note that the boundary of any complementary region is a path in the graph of the curves in $S$. If $\bigcup S$ is not connected, regluing the complementary regions, one gets a disconnected surface.
\end{proof}

\begin{definition}
    For a subsurface $Y$ of a hyperbolic surface, let $\#^p(Y)$ be the number of pairs of pants in a pants decomposition of $Y$.
\end{definition}

\begin{remark}
\label{pantscardinatlity}
    $\#^p(Y)=-\chi(Y)$, where $e$ is the Euler characteristic.
\end{remark}

\begin{lemma}
\label{maxhull}
Let $S$ be a $j$-system of $r$ curves, then
$$\#^p\ssh(S)\le j\binom{r}{2}.$$
\end{lemma}

\begin{proof}
    We calculate the Euler characteristic
    \begin{align*}
        \#^p\ssh(S)&=-\chi(\ssh(S))\\
        &=-V+E-F\\
        &=V-F\le V\le j\binom{r}{2}.
    \end{align*}
\end{proof}

\begin{definition}
    Suppose $[g,n]\neq[0,3]$. Let $M(g,n)$ be the maximum cardinality of a minimal filling set, and $m^j(g,n)$ the minimum cardinality of a filling $j$-system, on a $[g,n]$-surface, that is,
    $$M(g,n):=\max_{S \text{ minimal filling}}\#S,$$
    $$m^j(g,n):=\min_{S \text{ filling } j-\text{system}}\#S.$$
\end{definition}

\begin{lemma}
\label{minimalcardinality}
    We have the following estimate: $$m^j(g,n)>\sqrt{\frac{4g-4+2n}{j}}.$$
\end{lemma}

\begin{proof}
    Let $S=\{\gamma_1,\cdots,\gamma_r\}$ be a $j$-system that is filling a $[g,n]$-surface $X$, then $\ssh(S)=X$. By Remark \ref{pantscardinatlity} and Lemma \ref{maxhull}, we have
    $$2g-2+n\le j\binom{r}{2}.$$ Therefore, $$r>\sqrt{\frac{4g-4+2n}{j}}.$$
\end{proof}

\begin{remark}
    For other estimates, one can refer to \cite{anderson2011small} and \cite{fanoni2015filling}.
\end{remark}

\begin{lemma}
\label{iteration}
Suppose $[g,n](X)\neq[0,2]$ or $[0,3]$, then there exists a proper $[g',n']$-subsurface of $X$, such that
$$M(g,n)\le 1+M(g',n').$$
\end{lemma}

\begin{proof}
    Let $S=\{\gamma_1,\cdots,\gamma_r\}$ be a minimal filling set such that $r=M(g,n)$. Set $Y_1=\ssh(S\setminus\{\gamma_r\})$, then $Y_1\subsetneqq X$ by minimality. To show the minimality of $S\setminus\{\gamma_r\}$, we remove any curve, say $\gamma_{r-1}$, and set $Y_2=\ssh(S\setminus\{\gamma_{r-1},\gamma_r\})$. Note that $Y_2\subsetneqq Y_1$, otherwise we have
    \begin{align*}
        &\ssh(S\setminus\{\gamma_{r-1}\})=\ssh(S\setminus\{\gamma_{r-1},\gamma_r\},\gamma_r)\\
        =&\ssh(Y_2,\gamma_r)=\ssh(Y_1,\gamma_r)=X,
    \end{align*}
    which is contradictory to minimality of $S$. Let $[g',n']$ be the type of $Y_1$, then minimality of $S\setminus\{\gamma_r\}$ implies
    \begin{align*}
        M(g,n)=\#S=1+\#(S\setminus\{\gamma_r\})\le 1+M(g',n').
    \end{align*}
\end{proof}

\begin{remark}
    This process will not yield any $[0,3]$-subsurfaces.
\end{remark}

\begin{theorem}
\label{maxm}
    We have $$M(0,2)=1$$ and $$M(g,n)\le 3g+n.$$
\end{theorem}

\begin{proof}
    For a $[g,n]$-surface $X$, there are two types of maximal proper subsurfaces: $$[g-1,n+2], [g,n-1],$$ as long as the numbers are nonnegative and the Euler characteristic is negative. They are obtained by cutting $X$ along a non-separating or separating simple closed geodesic. Note that any proper subsurface can be obtained through a chain of maximal proper subsurfaces. Let $f(g,n)=3g+n$, then $f(Y)<f(X)$ for any proper subsurface $Y\subset X$. By Lemma \ref{iteration}, let $$Y_k\subsetneqq Y_{k-1}\subsetneqq\cdots\subsetneqq Y_1\subsetneqq X$$ be a sequence of subsurfaces, such that $Y_i$ is maximal in $Y_{i-1}$, and $Y_k$ is a $[0,2]$-subsurface. Therefore, $$3g+n=f(X)\ge f(Y_1)+1\ge \cdots \ge f(Y_k)+k= 2+k$$ and $$M(X)\le 1+M(Y_1)\le\cdots\le k+M(0,2)=k+1\le 3g+n-1.$$
    Note that $[0,3]$ is skipped in the descending process so we have $$M(g,n)\le 3g+n.$$
\end{proof}

\begin{definition}
The \textit{j-capacity} $\cp^j(Y)$ of a subsurface $Y$ is the maximum cardinality of a $j$-system on $Y$, that is,
    $$\cp^j(Y):=\max_{S j\text{-system}}\#S.$$
\end{definition}

\begin{theorem}
\label{maxcap}
We have the following estimate on $j$-capacity:
    $$\cp^j(g,n)\le M(g,n)+(2jM(g,n)(M(g,n)-1))^{jM(g,n)}.$$
\end{theorem}

\begin{proof}
    Note that given a filling $j$-system, the subset that has the smallest cardinality while filling is always a minimal filling subset. Let $S$ be a $j$-system on a $[g,n]$-surface $X$, then there exists a minimal filling subset $S_0\subset S$. By definition, there is $\#S_0\le M=M(X)$. Any $\gamma\in S\setminus S_0$ can be obtained from $S_0$ in the following sense.
    
    Let
    $$\delta_1,\delta_2,\cdots,\delta_l,$$ be all the elements in $S_0$ that intersect $\gamma$, ordered by an orientation of $\gamma$, where $\delta_i$'s are not necessarily distinct but each element of $S_0$ appears at most $j$ times. Consequently, $l\le jM$. Let $\gamma\setminus\cup S_0=\cup\gamma_i$, where $\gamma_i$ is a segment of $\gamma$ that connects $\delta_i$ and $\delta_{i+1}$. The segment $\gamma_i$ lives in a convex polygon or once-punctured or once-holed convex polygon that has segments of $\delta_i$ and $\delta_{i+1}$ cut by $S_0$ as two sides. Note that there are at most $jM(M-1)$ edges of $S_0$ considered as a graph.

    Given the initial and final point of $\gamma_i$, there are at most two homotopy classes for $\gamma_i$ as $\gamma$ is simple, therefore we get an upper bound of all topological possibilities of such $\gamma$: $$(jM(M-1))^l\cdot 2^{l},$$
    and thus $$\cp^j(g,n)\le M+(2jM(M-1))^{jM}.$$
\end{proof}

For the case of $j=1$, we have the polynomial result due to Przytycki.
\begin{theorem}[\cite{przytycki2015arcs}]
\label{prz}
    $$\cp^{1}(g,n)\le g(4g+2n-3)^2+2g+n-3.$$
\end{theorem}

\section{Essentiality of subsurfaces}
\label{essentialsection}

In the moduli space $\mathcal M_{1,1}$, the surface that is conformal to the Euclidean punctured square torus, as shown in Figure \ref{fig:(1,1)_1}, is a critical point for the systole function. There are two shortest geodesics on this surface, say $\alpha$ and $\beta$. By symmetry, one obtains $\nabla l_\alpha+\nabla l_\beta=0$. We may ask a simple question: Does this happen to any other surface where two curves have negative or collinear gradient vectors in the tangent space of its $\teich$ space?

A \textit{length gradient} is the (Weil-Petersson) gradient vector of the geodesic-length function associated to a (simple) closed geodesic on a surface (seen as a marking), that lives in the tangent space of the $\teich$ space.

On a hyperbolic surface $X$, take a pair of subsurfaces, say $Y_1$ properly contained in $Y_2$. Going from $Y_1$ to $Y_2$, it obviously increases the dimension of the tangent subspace of the Teichm\"uller space. Given a set $V$ of length gradients in $T_X\tgn$, one can restrict them to the tangent subspace $T_X^{Y_i}\tgn$ to get $V_i$, $i=1,2$. It is natural to expect $$\rank (V_1)<\rank (V_2),$$ for any, or at least for generic $X\in\tgn$.

Here is a slightly different setting. Start with a curve set $S_i$ on $Y_i$, $i=1,2$, with $S_1\subset S_2$, then possibly under the condition that $\ssh(S_i)=Y_i$, it should be expected that
\begin{equation*}
\label{rankleap}
    \rank\{\nabla l_\gamma\}_{\gamma\in S_1}<\rank\{\nabla l_\gamma\}_{\gamma\in S_2}. \tag{$\star$}
\end{equation*}

\begin{definition}
    (1) A subsurface is called \textit{essential} if each complementary region is a [0,3]-surface, otherwise it is called \textit{non-essential}. See Figure \ref{fig:Nonessential subsurface} for an example.
    
    (2) For a subsurface $Y\subset X$, the \textit{essential closure} $\overline Y$ of $Y$ is the maximal embedded subsurface of $X$ that $Y$ is essential in with $\partial\overline Y\subset\partial Y$. This should not be confused with topological closure. We write $\overline{\ssh}(\cdot)=\overline{\ssh(\cdot)}$.
\end{definition}

\begin{remark}
    Equivalently, a subsurface is essential if no complementary region contains a [1,1] or [0,4]-subsurface.
\end{remark}

\begin{lemma}
\label{essential closure upper bound}
    Let $Y$ be a subsurface, then $\#^p(\overline Y)<2\#^p(Y)$+2.
\end{lemma}

\begin{proof}
    Consider the complement $Y^c$. To get $\overline Y$, one attaches all $[0,3]$-components of $Y^c$ to $Y$ along its geodesic boundary components. Every attaching operation decreases the number of boundary components by 1 or 2 and increases $\#^p$ by 1. Therefore, there can be at most $n(Y)$ attaching operations, and thus $$\#^p(\overline Y)\le \#^p(Y)+n(Y)=2g(Y)+2n(Y)-2\le2\#^p(Y)+2.$$
\end{proof}

We show there is a leap of the rank of sets of length gradients, as in \eqref{rankleap}, when expanding a subsurface non-essentially.
\begin{lemma}
\label{essential}
    Let $S_1\subset S_2$ be two sets of curves on $X$, and $Y_i=\ssh(S_i)$, $i=1,2$. Suppose we have the conditions that
    
    (1) $Y_1\subsetneqq Y_2$,
    
    (2) $Y_1$ is not essential in $Y_2$.
    
    Then $$\rank\{\nabla l_\gamma\}_{\gamma\in S_1}<\rank\{\nabla l_\gamma\}_{\gamma\in S_2}.$$
\end{lemma}

Suppose that there are two curves $\alpha\in S_2\setminus S_1$ and $\delta\subset Y_2\setminus Y_1$ such that they intersect each other exactly once and non-orthogonally. By Kerckhoff's geodesic length-twist formula in \cite{kerckhoff1983nielsen}, we have $$\langle\nabla l_\alpha,\tau_\delta\rangle=\cos\theta(\alpha,\delta)\neq 0,$$ where $\tau_\delta$ is the tangent vector of the twisting path along $\delta$ at $X$. Intuitively, $\nabla l_\alpha$ reflects an extra dimension on top of the space spanned by $S_1$.

However, a random selection of $\alpha$ and $\delta$ does not guarantee the intersection condition. To find such a pair of curves, we utilize an auxiliary curve $\lambda$ and consider doing Dehn twists on $\delta$ along $\lambda$.

\begin{figure}[h]
    \includegraphics[width=7cm]{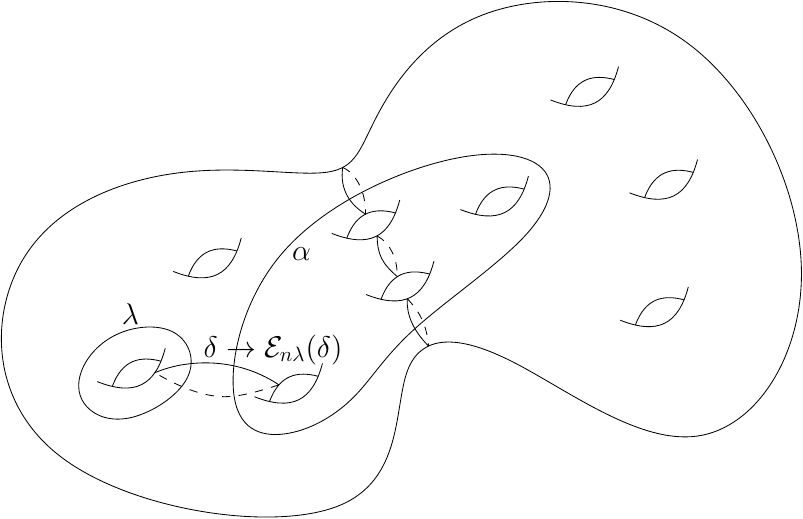}
    \caption{The right subsurface is non-essential; $\alpha$ and $\lambda$ are not necessarily disjoint}
    \label{fig:Nonessential subsurface}
\end{figure}

More precisely, we have the following lemma, where the conditions are shown in Figure \ref{fig:Nonessential subsurface}.

\begin{lemma}
\label{threecurves}
    Suppose $\alpha$, $\delta$, and $\lambda$ are three simple closed geodesics on a hyperbolic surface $X$, satisfying
    
    (1) $\delta$ and $\alpha$ intersect,
    
    (2) $\delta$ and $\lambda$ intersect.
    
    At time $t$ on the earthquake path $\mathcal E_\lambda$ through $X$, let $\alpha'(t)$ be the geodesic arc obtained from $\alpha$ by twisting the base surface along $\lambda$ by $t$; let $\delta(t)$ be equal to $\delta$ as markings; and let $\theta(t)$ be the angle of $\delta(t)$ and $\alpha'(t)$ at any given intersection. Then $\theta(t)$ is monotone along $\mathcal E_\lambda(t)$.
\end{lemma}

This can be seen as a corollary of Lemma 3.6 in \cite{kerckhoff1983nielsen}. We restate it as follows with our notation. Figure \ref{fig:Intersection along earthquake} is modified on Kerckhoff's original picture in the same paper. Here, on the universal cover of $X$, $\widetilde\lambda$ is the preimage of $\lambda$, and the Thurston's earthquake is realized by shearing the components complementary to $\widetilde\lambda$ along $\widetilde\lambda$, where a given component is fixed (e.g. the one containing $\widetilde\delta_0$). In the picture, $\widetilde\delta_i$'s are segments that connect consecutive leaves of $\widetilde\lambda$ to give $\delta(t)$ on $\mathcal E_\lambda(t)$ through $X$ and $\widetilde\alpha'$ is a lift of $\alpha$. Let $\overline\delta(t)$ be the geodesic with endpoints $\lim_{n\to\pm\infty}{\widetilde\delta_n}$. $\theta$ is an intersection angle of $\overline\delta(t)$ and $\widetilde\alpha'$.

\begin{lemma}[\cite{kerckhoff1983nielsen}]
    The endpoints of $\overline\delta(t)$ move strictly to the left when $t$ increases.
\end{lemma}

\begin{figure}[h]
    \includegraphics[width=7cm]{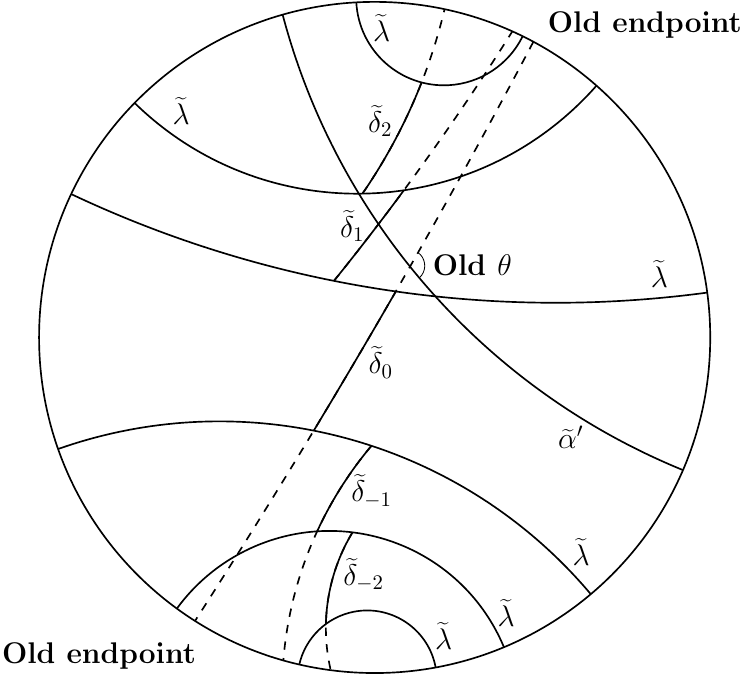}
    \caption{New endpoints are to the left to the old ones, meaning that $\overline\delta(t)$ rotates counterclockwise}
    \label{fig:Intersection along earthquake}
\end{figure}

\begin{proof}[Proof of Lemma \ref{essential}]
    Let $Z$ be a connected component of $Y_2\setminus Y_1$ that contains a $[1,1]$ or $[0,4]$-subsurface. Note that for any geodesic or orthogeodesic on $Z$, there exists a simple closed geodesic that intersects it. Since $Y_2\supsetneqq Y_1$, pick $\alpha\in S_2$ crossing $Z$, and let $\delta$ be a simple closed geodesic on $Z$ that intersects $\alpha$. Pick a simple closed geodesic $\lambda$ on $Z$ that intersects $\delta$. The conditions in Lemma \ref{threecurves} are satisfied.
    
    Let $\theta_i(t)\in(0,\pi)$ be the $i$-th intersection angle of $\alpha'(t)$ and $\delta(t)$, measured from $\alpha'(t)$ to $\delta(t)$, as in Lemma \ref{threecurves}, then the $\theta_i$'s have the same monotonicity in $t$ along the earthquake path $\mathcal E_\lambda(t)$ through $X$. Therefore $\sum\cos\theta_i(t)$ is monotone. Note that there are only finitely many zeros of $\sum\cos\theta_i(t)$, so there exists an integer $t=n$ such that $$\langle\nabla l_\alpha,\tau_{\delta(n)}\rangle=\sum\cos\theta_i\neq0$$ at $X$. On the other hand, $\langle\nabla l_\gamma,\tau_{\delta(n)}\rangle=0$ for any $\gamma\in S_1$. The lemma follows.
\end{proof}

The proof also implies the following statement.
\begin{lemma}
\label{extra rank}
    Let $Y$ be a subsurface of $X$, and $\gamma$ a simple closed geodesic on $X$. Suppose $\gamma\not\subset\overline Y$, then $\nabla l_\gamma\not\in T_X^Y\mathcal T$, the tangent subspace given by $Y$ of the Teichm\"uller space of $X$ at $X$.
\end{lemma}

\begin{theorem}
\label{induction}
Given $j\ge0$, for any $k\ge0$, there exists $r_k>0$ such that, with finitely many exceptions on $(g,n)$, for any $j$-system $S$ with $\#S\ge r_k$, on any surface $X$ in any $\teich$ space $\tgn$, there is
$$\rank\{\nabla\gamma\}_{\gamma\in S}\ge k.$$
\end{theorem}

\begin{proof}
    We prove the theorem by induction.
    
    The case of $k=0$ is trivial. The case of $k=1$ is also trivial with $r_1=1$.
    
    For any $j$-system $S_k$ of $r_k$ curves on $X$, when $-\chi(X)$ is large, we have the inductive assumption $$\text{rank}\{\nabla l_\gamma\}_{\gamma\in S_k}\ge k.$$
    
    Note that $\#^p\ssh(S_k)$ is bounded from above in $r_k$ by Lemma \ref{maxhull}. By Lemma \ref{essential closure upper bound}, there exists $\chi_k=\chi_k(r_k)$ such that $\ssh(S_k)$ is not essential in any surface $X$ when $-\chi(X)>-\chi_k$. By Theorem \ref{maxcap}, $\cp^j(\overline{\ssh}(S_k))$ is bounded in $r_k$ (and $j$) and uniformly in $S_k$.
    
    We pick $$r_{k+1}>\max_{S_k}\cp^j(\overline{\ssh}(S_k)).$$
    
    For any $j$-system $S_{k+1}$ of at least $r_{k+1}$ curves and any subset $S_k\subset S_{k+1}$ of $r_k$ curves, by definition of $j$-capacity, we have
    $$\ssh(S_{k+1})\supsetneqq\overline{\ssh}(S_k),$$
    and by Lemma \ref{essential} and Lemma \ref{extra rank},
    $$\text{rank}\{\nabla l_\gamma\}_{\gamma\in S_{k+1}}>\text{rank}\{\nabla l_\gamma\}_{\gamma\in S_k}\ge k.$$
    Therefore, $$\text{rank}\{\nabla l_\gamma\}_{\gamma\in S_{k+1}}\ge k+1.$$
    Induction completes.
\end{proof}

\section{Critical points of the Morse functions}
\label{systole section}
Recall that given a hyperbolic surface $X$, $S(X)$ denotes the set of shortest closed geodesics on it. Later we may extend the definition of $S(Y)$ to a surface $Y$ with boundary, or a subsurface $Y$, by including boundary curves when looking for the shortest.

As a set of curves, $S(X)$ satisfies certain properties for combinatorial and geometric reasons. We say that a curve on a surface bounds two cusps, if they together form a pair of pants.
\begin{lemma}
\label{intersection=2}
    Suppose $\gamma_1,\gamma_2\in S(X)$, then $i(\gamma_1,\gamma_2)\le 2$, namely, $S(X)$ is a 2-system. If $i(\gamma_1,\gamma_2)=2$, then at least one of them bounds two cusps.
\end{lemma}

This is a well-known fact about systole. For proof, one may see \cite{fanoni2016systoles}.
\begin{remark}
    $S(X)$ is a 1-system when $n(X)=0,1$.
\end{remark}

\begin{corollary}
\label{separatingsystole}
    Suppose $S(X)$ is filling, then $\gamma\in S(X)$ is separating if and only if it bounds two cusps.
\end{corollary}

\begin{lemma}
\label{sametwocusps}
    If two distinct geodesics $\gamma_1$ and $\gamma_2$ bound the same two cusps on a surface $X$, then $i(\gamma_1,\gamma_2)\ge4$.
\end{lemma}

\begin{proof}
    For $i=1,2$, since $\gamma_i$ bounds two cusps, say $p$ and $q$, either separates the surface into two parts. Consider $p$ and $q$ as two marked points on the surface for topological constructions. Let $X_i$ denote the closed $[0,3]$-subsurface bounded by $\gamma_i$ that contains the cusps $p$ and $q$, then $p,q\in X_1\cap X_2$. Note that $X_1\cap X_2$ is not path connected, as $p$ and $q$ cannot be joined by a path in the interior; otherwise, both $X_1$ and $X_2$ contract to that path. Since the boundary of the path component containing $p$ or $q$ is contributed by both $\gamma_1$ and $\gamma_2$, it contains at least two intersections of $\gamma_1$ and $\gamma_2$. The lemma follows.
\end{proof}

\begin{corollary}
\label{same two cusps}
    For distinct $\gamma_1,\gamma_2\in S(X)$, they cannot bound the same two cusps.
\end{corollary}

Per Remark \ref{remark rank statement}, we apply Theorem \ref{induction} onto $S(X)$ which is a 2-system, for a hyperbolic surface $X$ that is large enough. We have

\begin{theorem}
\label{main theorem down}
    For any $k\ge0$, with finitely many exceptions on $(g,n)$, all critical points have Morse index greater than $k$ in $\mgn$, for the systole function or the $\syst$ functions.
\end{theorem}

\begin{remark}
    This implies that there are only finitely many critical points of index at most $k$ in $\bigsqcup_{(g,n)}\mgn$, for the systole function or the $\syst$ functions. The critical points of index at most $k$ for the $\syst$ functions live in the Deligne-Mumford boundary $\partial\mgn$.
\end{remark}

\begin{proof}
    Let $p$ be a critical point for the systole function, and $p_T$ a critical point for $\syst$, as in Theorem \ref{syst}. Note that $$\ind_{\syst}(p_T)=\ind_{\sys}(p)=\rank\{\nabla l_\gamma\}_{\gamma\in S(p)}.$$
    This theorem follows from Theorem \ref{induction}.
\end{proof}

\section{Bounding the index gap}
\label{logloglowerbound}

In this section, we compute a lower bound for the lowest index of the systole function on $\mgn$, and the $\syst$ functions on $\mgn\subset\mgnb$.

Let the \textit{kissing number} $\kiss(X)$ of a surface $X$, possibly with boundary, be the maximum cardinality of $S(X)$. We use Przytycki's result stated in Theorem \ref{prz} to bound the $\kiss(Y)$ of a subsurface $Y$. Note that in the case when $Y$ has no boundary, an upper bound is given in \cite{fanoni2016systoles}.

\begin{lemma}
\label{kissingnumbersubsurface}
    For a $[g,n]$-surface $Y$, we have $$\#S(Y)\le g(4g+2n-3)^2+2g+n-3+\binom{n}{2}.$$
\end{lemma}

\begin{proof}
    Let $S^2(X)$ be the subset of $S(X)$ consisting of curves that bound two cusps, and $S^1(X)$ the complement. Note that $S^1(X)$ is a 1-system, and by Lemma \ref{intersection=2} and Lemma \ref{same two cusps}, we have $\#S^2(X)\le\binom{n}{2}$.
\end{proof}

\begin{theorem}
    There exists a universal constant $C$ such that the lowest index of the systole function and $\syst$ functions on $\mgn$ is at least $C\log\log(g+n)$.
\end{theorem}

\begin{proof}
We use the same induction argument as in the proof of Theorem \ref{induction}. By Corollary \ref{rank=2} in the next section, set $r_2=3$. Suppose $r_k$ is given as in Theorem \ref{induction}. We will pick $$r_{k+1}>\max_{S_k\subset S(X), \#S_k=r_k}\kiss(\overline{\ssh}(S_k)).$$

Note that for any $S_k\subset S(X)$ with $\#S_k=r_k$, we have

$$\#^p\ssh(S_k)\le 2\binom{r_k}{2}$$

and

$$\#^p\overline{\ssh}(S_k)\le 2\#^p\ssh(S_k)+2.$$

For $r_k>3$, by Lemma \ref{kissingnumbersubsurface}, we have $$\max_{S_k}\kiss(\overline{\ssh}(S_k))<\frac{1}{2}(r_k^2+2)(4r_k^2+5)^2+2r_k^2+1+\frac{1}{2}(r_k^2+2)^2<Cr_k^6.$$

We set $$r_{k+1}=Cr_k^6,$$ for $k\ge 2$, then $$r_k=C(Cr_2)^{6^{k-2}}.$$

On the other hand, for a critical point $X$, by Lemma \ref{minimalcardinality}, we have $$\#S(X)\ge\sqrt{-2\chi(X)},$$ and then

$$\ind(X)\ge\max\{k|\#S(X)\ge r_k\}\ge\max\{k|\sqrt{-2\chi(X)}\ge r_k\}.$$

The theorem follows.
\end{proof}

\section{Classification of Low Index Critical Points}
\label{classification}

We give more explicit results on Theorem \ref{induction} for some small $k$ in the settings of shortest geodesics.

\begin{lemma}
\label{rank=2}
    Suppose $(g,n)(X)\neq (1,1),(0,4)$, then for distinct $\gamma_1,\gamma_2\in S(X)$, $$\rank\{\nabla l_1,\nabla l_2\}=2.$$  
\end{lemma}

\begin{proof}
    It is trivial that $$\rank\{\nabla l_1\}=\rank\{\nabla l_2\}=1.$$

    Consider $\gamma_1=\ssh(\gamma_1)$. Note that $\gamma_1$ is non-essential in any hyperbolic surface $X$ except when $(g,n)(X)=(1,1)$ or $(0,4)$. In any non-exceptional case, following Lemma \ref{extra rank}, we have $$\rank\{\nabla l_1,\nabla l_2\}>\rank\{\nabla l_1\}=1,$$ that is, $\rank\{\nabla l_1,\nabla l_2\}=2$.
\end{proof}

We recall a direct corollary of Riera's formula.
\begin{lemma}[\cite{riera2005formula}, \cite{wolpert1987geodesic}]
\label{lengthgradientdisjoint}
Let $\gamma_i,\gamma_j$ be disjoint or identical geodesics, then $$\left\langle\nabla l_i,\nabla l_j\right\rangle_{\textit{WP}}>0.$$
\end{lemma}

\begin{lemma}
\label{rank=3}
    Suppose $(g,n)(X)\neq (1,1),(0,4),(1,2),(0,5)$, then for distinct $\gamma_1,\gamma_2,\gamma_3\in S(X)$, $$\rank\{\nabla l_1,\nabla l_2,\nabla l_3\}=3.$$
\end{lemma}

\begin{proof}
    If the union of $\gamma_1,\gamma_2,\gamma_3$ is not connected, using Lemma \ref{lengthgradientdisjoint}, it reduces to the case of two curves as above. Suppose $\gamma_1$ intersects $\gamma_2$ and $\gamma_3$, and let $Y_{12}=\ssh\{\gamma_1,\gamma_2\}\subset X$.
    
    For the case of $\gamma_3\in Y_{12}$, given the equal length of all $\gamma_i$'s, $Y_{12}$ as a $(1,0,1)$ or $(0,3,1)$-subsurface is determined and has a $\mathbb Z/3$ rotational symmetry, then $\{\nabla l_1,\nabla l_2,\nabla l_3\}$ has rank 2 when projected onto the 2-dimensional tangent subspace at $X$ of $\mathcal T(X)$ given by $Y_{12}$ (boundary not considered). Let $\delta$ be the geodesic boundary of $Y_{12}$, then $$\langle\nabla l_i,\nabla l_\delta\rangle>0$$ by Lemma \ref{lengthgradientdisjoint}. Therefore, $$\rank\{\nabla l_1,\nabla l_2,\nabla l_3\}=3.$$
    
    Suppose now $\gamma_3\not\in Y_{12}$. 
        
    If $i(\gamma_1,\gamma_2)=1$, then $[g,n](Y_{12})=[1,1]$, and $Y_{12}$ is non-essential in any $X$ when $(g,n)(X)\neq(1,2)$.
    
    If $i(\gamma_1,\gamma_2)=2$, then $[g,n](Y_{12})=[0,4]$, and at least one of $\gamma_1$ and $\gamma_2$ bounds two cusps, so $Y_{12}$ has at least two punctures. $Y_{12}$ is non-essential in any $X$ when $(g,n)(X)\neq(1,3),(0,5)$ or $(0,6)$.
    
    For any topological type of $X$ other than the ones mentioned above, the conclusion follows from the previous lemma on rank 2 and Lemma \ref{essential}. Furthermore, the following two topological types can also be excluded, where we do not assume that $\gamma_1$ intersects $\gamma_2$ and $\gamma_3$.
    
    Suppose $(g,n)(X)=(1,3)$. Consider the subset $S_c=\{\gamma_i,\gamma_i \text{ bounds 2 cusps}\}$.
    
    If $\#S_c=0$ or 1, suppose $\gamma_1,\gamma_2\not\in S_c$, then $\gamma_1$ and $\gamma_2$ are non-separating by Corollary \ref{separatingsystole}. If $\gamma_1$ and $\gamma_2$ intersect, then $Y_{12}=\ssh\{\gamma_1,\gamma_2\}$ is non-essential in $X$, as the complement is a $[0,4]$-subsurface. If $\gamma_1$ and $\gamma_2$ are disjoint, then $Y_{12}$ is non-essential, as a component of the complement is a $[0,4]$-subsurface.
    
    If $\#S_c=2$ or 3, suppose $\gamma_1,\gamma_2\in S_c$, then each of $\gamma_1$ and $\gamma_2$ bounds a pair of cusps, which by Lemma \ref{same two cusps}, are not the same. Therefore, $Y_{12}$ is a $(0,3,1)$-subsurface, and is non-essential in $X$.

    In either case, the rank result follows from Lemma \ref{lengthgradientdisjoint}.
    
    Suppose $(g,n)(X)=(0,6)$. Since any closed geodesic is separating, and in any pair of curves in $S(X)$, by Lemma \ref{intersection=2}, at least one bounds two cusps, then at least two of the $\gamma_i$'s bound a pair of cusps, say $\gamma_1$ and $\gamma_2$. Then $\ssh\{\gamma_1,\gamma_2\}$ is a $(0,3,1)$-subsurface and is therefore non-essential in $X$. As $\gamma_3\not\subset\overline{\ssh}\{\gamma_1,\gamma_2\}$, we have $$\rank\{\nabla l_1,\nabla l_2,\nabla l_3\}\ge\rank\{\nabla l_1,\nabla l_2\}+1=3.$$

\end{proof}

The two lemmas above imply that no critical points of any of the indices exist in any non-exceptional moduli spaces. For the exceptional cases, the critical points for the systole function are known thanks to Schmutz Schaller in \cite{schaller1999systoles}. We give the classification of critical points for the systole function, of low indices, namely, 0, 1 and 2. For each figure in the following theorems, there exists a unique surface where the colored curves are exactly the shortest geodesics, with the given information on intersection or symmetry.

\begin{theorem}
    All critical points of index 0 are shown in Figure \ref{fig:(0,3)_0}.
\end{theorem}

\begin{figure}[h]
    \centering
    \includegraphics[width=5cm]{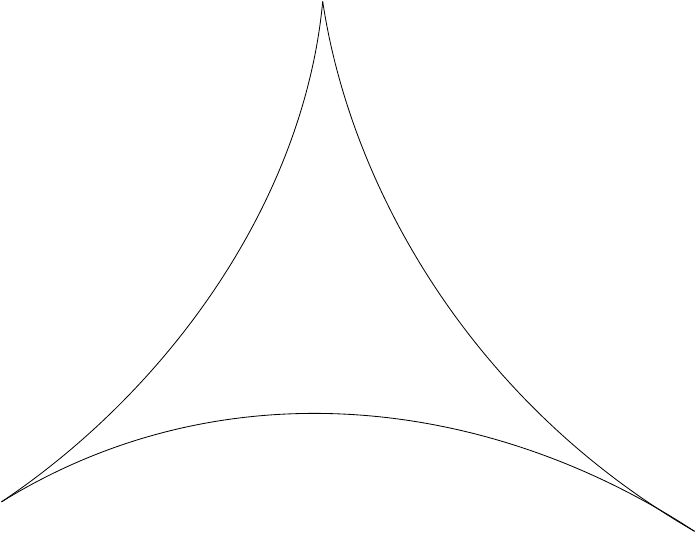}
    \caption{$(g,n)=(0,3)$}
    \label{fig:(0,3)_0}
\end{figure}

\begin{theorem}
    All critical points of index 1 are shown in Figure \ref{fig:(1,1)_1}, \ref{fig:(0,4)_1}.
\end{theorem}

\begin{figure}[h]
    \centering
    \includegraphics[width=5cm]{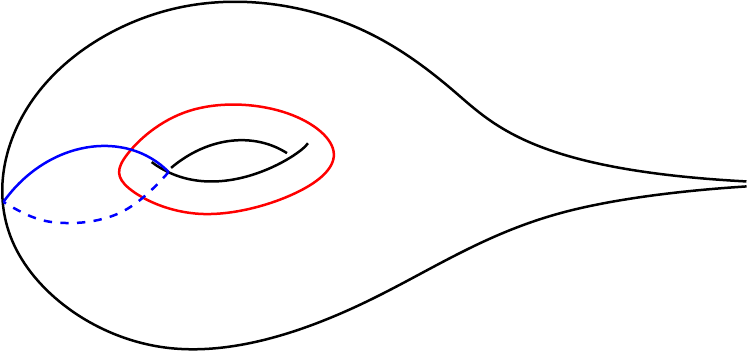}
    \caption{$(g,n)=(1,1),\#S(X)=2$, $\frac{\pi}{2}$-intersection}
    \label{fig:(1,1)_1}
\end{figure}

\begin{figure}[h]
    \centering
    \includegraphics[width=3.5cm]{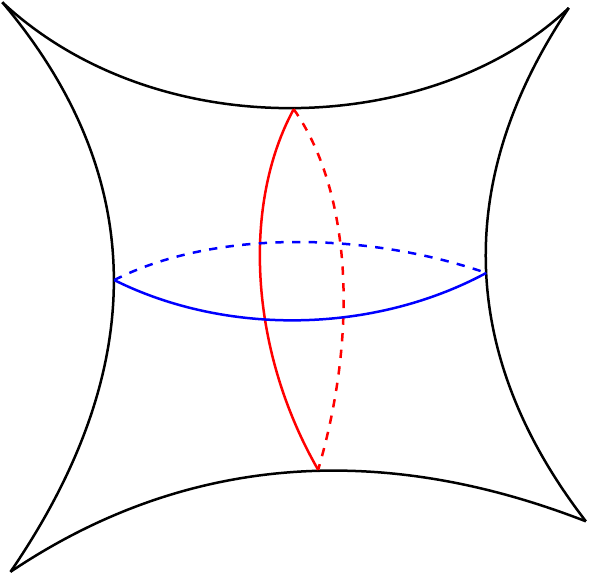}
    \caption{$(g,n)=(0,4),\#S(X)=2$, $\frac{\pi}{2}$-intersection}
    \label{fig:(0,4)_1}
\end{figure}

\begin{theorem}
    All critical points of index 2 are shown in Figure \ref{fig:(1,1)_2}, \ref{fig:(0,4)_2}, \ref{fig:(1,2)_2a}, \ref{fig:(1,2)_2b}, \ref{fig:(0,5)_2}.
\end{theorem}

\begin{figure}[h]
    \centering
    \includegraphics[width=5cm]{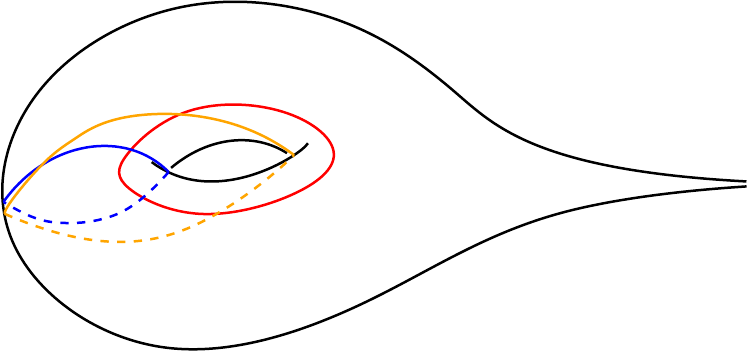}
    \caption{$(g,n)=(1,1),\#S(X)=3$}
    \label{fig:(1,1)_2}
\end{figure}

\begin{figure}[h]
    \centering
    \includegraphics[width=3.5cm]{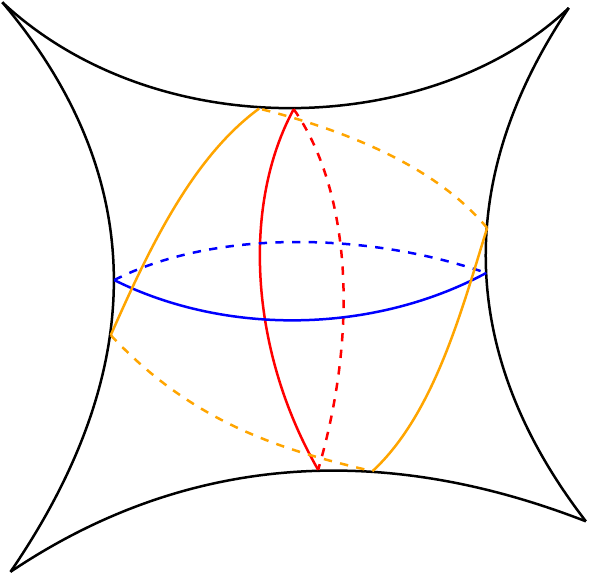}
    \caption{$(g,n)=(0,4),\#S(X)=3$}
    \label{fig:(0,4)_2}
\end{figure}

\begin{figure}[h]
    \centering
    \includegraphics[width=6cm]{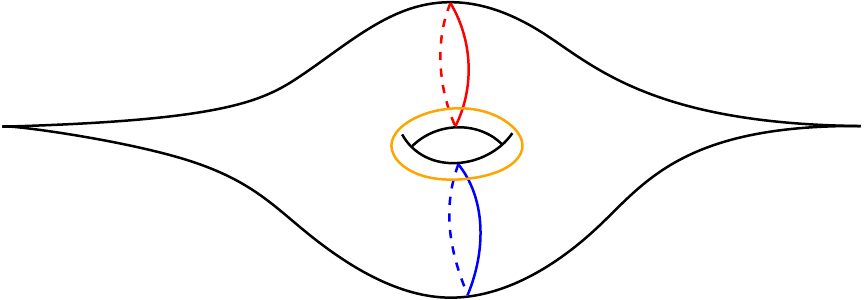}
    \caption{$(g,n)=(1,2),\#S(X)=3$, $\frac{\pi}{2}$-intersections}
    \label{fig:(1,2)_2a}
\end{figure}

\begin{figure}[h]
    \centering
    \includegraphics[width=4cm]{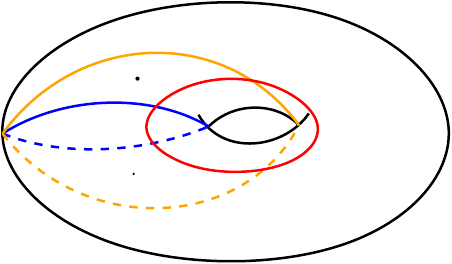}
    \caption{$(g,n)=(1,2),\#S(X)=3$, $\mathbb Z/2$ rotational and $\mathbb Z/3$ permutational symmetry}
    \label{fig:(1,2)_2b}
\end{figure}

\begin{figure}[h]
    \centering
    \includegraphics[width=4cm]{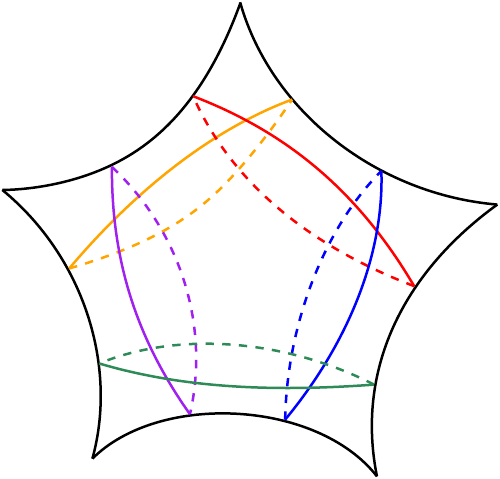}
    \caption{$(g,n)=(0,5),\#S(X)=5$, $\mathbb Z/5$ rotational symmetry}
    \label{fig:(0,5)_2}
\end{figure}

\clearpage
\bibliographystyle{alpha}
\bibliography{main}

\begin{thebibliography}{MRT14}

\bibitem[ABG19]{aougab2019packing}
Tarik Aougab, Ian Biringer, and Jonah Gaster.
\newblock {Packing curves on surfaces with few intersections}.
\newblock {\em International Mathematics Research Notices}, 2019(16):5205--5217, 2019.

\bibitem[Akr03]{akrout2003singularites}
Hugo Akrout.
\newblock {Singularit{\'e}s topologiques des systoles g{\'e}n{\'e}ralis{\'e}es}.
\newblock {\em Topology}, 42(2):291--308, 2003.

\bibitem[APP11]{anderson2011small}
James~W. Anderson, Hugo Parlier, and Alexandra Pettet.
\newblock {Small filling sets of curves on a surface}.
\newblock {\em Topology and its Applications}, 158(1):84--92, 2011.

\bibitem[Bus10]{buser2010geometry}
Peter Buser.
\newblock {\em {Geometry and spectra of compact Riemann surfaces}}.
\newblock Springer Science \& Business Media, 2010.

\bibitem[Che23]{chen2023c}
Changjie Chen.
\newblock {Morse theory on moduli of curves}.
\newblock {\em arXiv preprint arXiv:2307.16119}, 2023.

\bibitem[FB20]{bourque2020hyperbolic}
Maxime Fortier~Bourque.
\newblock {Hyperbolic surfaces with sublinearly many systoles that fill}.
\newblock {\em Commentarii Mathematici Helvetici}, 95(3), 2020.

\bibitem[FP15]{fanoni2015filling}
Federica Fanoni and Hugo Parlier.
\newblock {Filling sets of curves on punctured surfaces}.
\newblock {\em arXiv preprint arXiv:1508.03503}, 2015.

\bibitem[FP16]{fanoni2016systoles}
Federica Fanoni and Hugo Parlier.
\newblock {Systoles and kissing numbers of finite area hyperbolic surfaces}.
\newblock {\em Algebraic \& Geometric Topology}, 15(6):3409--3433, 2016.

\bibitem[Gre18]{greene2018curves}
Joshua~Evan Greene.
\newblock {On curves intersecting at most once}.
\newblock {\em arXiv preprint arXiv:1807.05658}, 2018.

\bibitem[IT12]{imayoshi2012introduction}
Yoichi Imayoshi and Masahiko Taniguchi.
\newblock {\em {An introduction to Teichm{\"u}ller spaces}}.
\newblock Springer Science \& Business Media, 2012.

\bibitem[JMM96]{juvan1996systems}
Martin Juvan, Aleksander Malni{\v{c}}, and Bojan Mohar.
\newblock {Systems of curves on surfaces}.
\newblock {\em Journal of Combinatorial Theory, Series B}, 68(1):7--22, 1996.

\bibitem[Ker83]{kerckhoff1983nielsen}
Steven~P Kerckhoff.
\newblock {The Nielsen realization problem}.
\newblock {\em Annals of mathematics}, pages 235--265, 1983.

\bibitem[Mil16]{milnor2016morse}
John Milnor.
\newblock {Morse Theory}.
\newblock Princeton University Press, 2016.

\bibitem[Mor59]{morse1959topologically}
Marston Morse.
\newblock {Topologically non-degenerate functions on a compact $n$-manifold $M$}.
\newblock {\em Journal d’Analyse Math{\'e}matique}, 7(1):189--208, 1959.

\bibitem[MRT14]{malestein2014topological}
Justin Malestein, Igor Rivin, and Louis Theran.
\newblock {Topological designs}.
\newblock {\em Geometriae Dedicata}, 168(1):221--233, 2014.

\bibitem[Prz15]{przytycki2015arcs}
Piotr Przytycki.
\newblock {Arcs intersecting at most once}.
\newblock {\em Geometric and Functional Analysis}, 25:658--670, 2015.

\bibitem[Rie05]{riera2005formula}
Gonzalo Riera.
\newblock {A formula for the Weil-Petersson product of quadratic differentials}.
\newblock {\em Journal d’Analyse Math{\'e}matique}, 95(1):105--120, 2005.

\bibitem[SS99]{schaller1999systoles}
Paul Schmutz~Schaller.
\newblock {Systoles and topological Morse functions for Riemann surfaces}.
\newblock {\em Journal of Differential Geometry}, 52(3):407--452, 1999.

\bibitem[Wol87]{wolpert1987geodesic}
Scott~A Wolpert.
\newblock {Geodesic length functions and the Nielsen problem}.
\newblock {\em Journal of Differential Geometry}, 25(2):275--296, 1987.

\end{thebibliography}

\end{document}